\documentclass[12pt]{amsart}

\usepackage{amssymb}
\usepackage[all]{xy}
\usepackage{graphics}
\usepackage{color}
\usepackage[T1]{fontenc}
\usepackage[colorlinks]{hyperref}
\usepackage{enumitem}

\newcommand{\B}[1]{\mathbb{#1}}

\usepackage[colorinlistoftodos,prependcaption]{todonotes}

\newtheorem{theorem}[subsection]{Theorem}

\newtheorem{lemma}[subsection]{Lemma}
\newtheorem{proposition}[subsection]{Proposition}

\newtheorem{conjecture}[subsection]{Conjecture}
\theoremstyle{definition}

\newtheorem{example}[subsection]{Example}

\newtheorem{notation}[subsection]{Notation}
\theoremstyle{remark}
\newtheorem{remark}[subsection]{Remark}

\numberwithin{figure}{section}
\numberwithin{table}{section}
\numberwithin{equation}{section}

\def\PGL{\operatorname{PGL}}
\def\SL{\operatorname{SL}}
\def\SO{\operatorname{SO}}

\newcommand{\OP}{\operatorname}

\def\rank{\OP{rank}}

\subjclass[2010]{Primary 20G15; Secondary 20B07, 22E15}

\date{February 28, 2022}

\begin{document}

\title{Uniform boundedness for algebraic groups and Lie groups}
\author{Jarek K\k{e}dra}
\address{University of Aberdeen and University of Szczecin}
\email{kedra@abdn.ac.uk}
\author{Assaf Libman}
\address{University of Aberdeen}
\email{a.libman@abdn.ac.uk}
\author{Ben Martin}
\address{University of Aberdeen}
\email{b.martin@abdn.ac.uk}

\begin{abstract}
 Let $G$ be a semisimple linear algebraic group over a field $k$ and let $G^+(k)$ be the subgroup generated by the subgroups $R_u(Q)(k)$, where $Q$ ranges over all the minimal $k$-parabolic subgroups $Q$ of $G$.  We prove that if $G^+(k)$ is bounded then it is uniformly bounded.  Under extra assumptions we get explicit bounds for $\Delta(G^+(k))$: we prove that if $k$ is algebraically closed then $\Delta(G^+(k))\leq 4 \rank(G)$, and if $G$ is split over $k$ then $\Delta(G^+(k))\leq 28 \rank(G)$.  We deduce some analogous results for real and complex semisimple Lie groups.
\end{abstract}

\maketitle


\section{Introduction}

In this paper we investigate the boundedness behaviour of a semisimple linear algebraic group $G$ over an infinite field $k$.  (For definitions of boundedness and related notions, see Section~\ref{sec:bddness}.)  If $k= {\mathbb R}$ then $G$ is a semisimple Lie group, and it is well known that $G$ is compact in the real topology if and only if it is anisotropic.  The authors showed in \cite[Thm.\ 1.2]{KLM1} that if $G$ is compact then $G$ is bounded but is not uniformly bounded; on the other hand, if $G$ has no simple compact factors then $G$ is uniformly bounded.  Motivated by this, we make the following conjecture.

\begin{conjecture}
\label{conj:main}
 Let $G$ be a semisimple linear algebraic group over an infinite field $k$.
 Then $G^+(k)$ is uniformly bounded.
\end{conjecture}

\noindent Here $G^+(k)$ denotes the subgroup of $G(k)$ generated by the subgroups $R_u(Q)(k)$, where $Q$ ranges over the minimal $k$-parabolic subgroups of $G$.  If $k= \overline{k}$ then $G^+(k)= G(k)$, while if $G$ is anisotropic over $k$ then $G^+(k)= 1$.  If $G$ has no anisotropic $k$-simple factors then $G^+(k)$ is dense in $G$.  Note that a finite group is clearly uniformly bounded, so Conjecture~\ref{conj:main} and the other results below all hold trivially for a semisimple linear algebraic group over a finite field $k$.

We make some steps towards proving the conjecture.

\begin{theorem}
\label{thm:bdd_unifbdd}
 Let $G$ be a semisimple linear algebraic group over an infinite field $k$, and suppose $G(k)= G^+(k)$.  Then $G(k)$ is finitely normally generated.   Moreover, if $G(k)$ is bounded then $G(k)$ is uniformly bounded.
\end{theorem}

We want to give explicit bounds for $\Delta(G)$ in terms of Lie-theoretic quantities such as $\rank G$ and $\dim G$.  We can do this in some special cases.  The first improves the bound $4\dim G$ from \cite[Thm.~4.3]{KLM1}.

\begin{theorem}
\label{thm:algclosed}
 Let $G$ be a semisimple linear algebraic group over an algebraically closed field $k$.  Then $\Delta(G(k))\leq 4\rank G$.  
\end{theorem}

\begin{theorem}
\label{thm:split}
 Let $G$ be a split semisimple linear algebraic group over an infinite field $k$.  Then $\Delta(G^+(k))\leq 28\rank G$.  
\end{theorem}

\noindent When $k= {\mathbb R}$, we get the following result.

\begin{theorem}
\label{thm:real_Lie}
 Let $H$ be a real semisimple linear algebraic group with no compact simple factors.  Then $H$ is uniformly bounded.  Moreover, if $H$ is split then $\Delta(H)\leq 28\rank G$.
\end{theorem}

\noindent When $k= {\mathbb C}$, we get the following result.

\begin{theorem}
\label{thm:complex_Lie}
 Let $H$ be a complex semisimple linear algebraic group.  Then $H$ is uniformly bounded and $\Delta(H)\leq 4\rank G$.
\end{theorem}

The idea of the proofs is as follows.  First we prove Theorem~\ref{thm:algclosed} (Section~\ref{sec:algclosed}); the new ingredient is that we work in the quotient variety $G/{\rm Inn}(G)$ rather than in $G$, which allows us to improve on the bound in \cite[Thm.~4.3]{KLM1}.  A key result underpinning our theorems for non-algebraically closed $k$ is Proposition~\ref{prop:gettingU}.  We prove this in Section~\ref{sec:isotropic} and deduce Theorem~\ref{thm:bdd_unifbdd}.  When $G$ is split we obtain  Theorem~\ref{thm:split} from Proposition~\ref{prop:gettingU} and the Bruhat decomposition; see Section~\ref{sec:split}.
In Section~\ref{sec:Lie} we prove Theorems~\ref{thm:real_Lie} and \ref{thm:complex_Lie}.

\subsection*{Acknowledgements} This work was funded by Leverhulme Trust Research Project Grant RPG-2017-159.

\section{Boundedness and uniform boundedness}
\label{sec:bddness}

A conjugation-invariant norm on a group $H$ is a non-negative function $\| \ \| \colon H \to \B R$ such that $\| \ \|$ is constant on conjugacy classes, $\|g\|=0$ if and only if $g=1$ and $\|gh\| \leq \|g\|+\|h\|$ for all $g,h\in H$.
The diameter of $H$, denoted $\|H\|$, is $\sup_{g \in H} \| g\|$.
A group $H$ is called \emph{bounded} if every conjugation-invariant norm has finite diameter.
In \cite{KLM1} we introduced two stronger notions of boundedness.  We briefly recall them now.

A subset $S \subseteq H$ is said to \emph{normally generate} $H$ if the union of the conjugacy classes of its elements generates $H$.
Thus, every element of $H$ can be written as a word in the conjugates of the elements of $S$ and their inverses.
Given $g \in H$, the length of the shortest such word that is needed to express $g$ is the \emph{word norm} of $g$ denoted $\|g\|_S$.
It is a conjugation-invariant norm on $H$.
The \emph{diameter} of $H$ with respect to this word norm is denoted $\|H\|_S$.
For every $n \geq 0$ we define
\[
B_S^H(n) = \{ g \in H \,|\, \|g\|_S \leq n\},
\]
the ball of radius $n$ (of all elements that can be written as a product of $n$ or fewer conjugates of the elements of $S$ and their inverses).
When there is no danger of confusion we simply write $B_S(n)$ (cf.\ Notation~\ref{notn:ball}).

We will use the following result \cite[Lem.\ 2.3]{KLM1} repeatedly: if $X, Y\subseteq H$ and $Y\subseteq B_X(m)$ then $B_Y(n)\subseteq B_X(mn)$.

We say that $H$ is {\em finitely normally generated} if it admits a finite normally generating set.  In this case we define
\begin{eqnarray*}
&& \Delta_k(H) = \OP{\sup} \{ \|H\|_S : \text{$S$ normally generates $H$ and $|S|\leq k$} \} \\
&& \Delta(H) = \OP{\sup} \{ \|H\|_S : \text{$S$ normally generates $H$ and $|S|< \infty$} \}.
\end{eqnarray*}
A finitely normally generated group $H$ is called \emph{strongly bounded} if $\Delta_k(H)<\infty$ for all $k$.
It is called \emph{uniformly bounded} if $\Delta(H)<\infty$.
Notice that $\Delta_k(H) \leq \Delta(H)$ for all $k\in {\mathbb N}$, so uniform boundedness implies strong boundedness.
It follows from \cite[Corollary 2.9]{KLM1} that strong boundedness implies boundedness.

\section{Linear algebraic groups}
\label{sec:LAG}

We recall some material on linear algebraic groups; see \cite{MR1102012} and \cite{MR2458469} for further details.  Below $k$ denotes an infinite field and $G$ denotes a semisimple linear algebraic $k$-group; we write $r$ for $\rank G$.  We adopt the notation of \cite{MR1102012}: we regard $G$ as a linear algebraic group over the algebraic closure $\overline{k}$ together with a choice of $k$-structure.  We identify $G$ with its group of $\overline{k}$-points $G(\overline{k})$.  If $H$ is any $k$-subgroup of $G$ then we denote by $H(k)$ the group of $k$-points of $H$.  More generally, if $C$ is any subset of $G$---not necessarily closed or $k$-defined---then we set $C(k)= C\cap G(k)$.  By \cite[V.18.3 Cor.]{MR1102012}, $G(k)$ is dense in $G$.

Fix a maximal split $k$-torus $S$ of $G$.  Let $L= C_G(S)$ and fix a $k$-parabolic subgroup $P$ such that $L$ is a Levi subgroup of $P$.  Set $U= R_u(P)$.  Then $P$ is a minimal $k$-parabolic subgroup of $G$, $L$ and $S$ are $k$-defined and $P$, $S$ are unique up to $G^+(k)$-conjugacy \cite[15.4.7]{MR2458469}.  Fix a maximal $k$-torus $T$ of $G$ such that $S\subseteq T$ and a (not necessarily $k$-defined) Borel subgroup $B$ of $G$ such that $T\subseteq B\subseteq P$.

\begin{notation}
\label{notn:ball}
 If $X\subseteq G^+(k)$ then we write $B_X(n)$ for $B_X^{G^+(k)}(n)$.
\end{notation}

\begin{lemma}
\label{lem:open_prod}
 Let $O, O'$ be nonempty open subsets of $G$.  For any $g\in G(k)$, there exist $h\in O(k)$ and $h'\in O'(k)$ such that $g=hh'$.
\end{lemma}

\begin{proof}
 Since $G$ is irreducible as a variety, $O^{-1}g\cap O'$ is an open dense subset of $G$.  Since $G(k)$ is dense in $G$, we can choose $h'\in (O^{-1}g)(k)\cap O'(k)$.  We can write $h'=h^{-1}g$ for some $h\in O(k)$.  This yields $g=hh'$, as required.
\end{proof}

For the rest of the section we assume that $G$ is split over $k$; then $S= T$ and $P= B$.  Let $\Psi_T$ denote the set of roots of $G$ with respect to $T$.  For $\alpha\in \Psi_T$, we denote by $U_\alpha$ the corresponding root group.  Let $\alpha_1,\ldots, \alpha_r$ be the base for the set of positive roots associated to $B$.  Note that $U_{\alpha_i}$ commutes with $U_{-\alpha_j}$ if $i\neq j$ because $\alpha_i- \alpha_j$ is not a root.  Let $U^-$ be the opposite unipotent subgroup to $U$ with respect to $T$.  Let $G_\alpha= \langle U_\alpha\cup U_{-\alpha}\rangle$ for $\alpha\in \Psi_T$; then $G_\alpha$ is $k$-isomorphic to either $\SL_2$ or $\PGL_2$.  Let $\alpha^\vee\colon {\mathbb G}_m\to G_\alpha$ be the coroot associated to $\alpha$.  The image $T_\alpha$ of $\alpha^\vee$ is $G_\alpha\cap T$, and this is a maximal torus of $G_\alpha$.

 We use the Bruhat decomposition for $G(k)$.  We recall the necessary facts \cite[Sec.\ V.14, Sec.\ V.21]{MR1102012}.  Fix a set $\widetilde{W}\subseteq N_G(T)(k)$ of representatives for the Weyl group; we denote by $n_0\in \widetilde{W}$ the representative corresponding to the longest element of $W$ (note that $n_0^2\in T(k)$ and $n_0Un_0^{-1}= U^-$).  The Bruhat decomposition $G= \bigsqcup_{n\in \widetilde{W}} BnB$ for $G$ yields a decomposition $G(k)= \bigsqcup_{n\in \widetilde{W}} B(k)nB(k)$ for $G(k)$ \cite[Thm.\ V.21.15]{MR1102012}.  The double coset $Bn_0B$ is open and $k$-defined.  The map $U\times B\to Bn_0B$, $(u,b)\mapsto un_0b$ is an isomorphism of varieties.  Hence if $g\in Bn_0B(k)$ then $g= un_0b$ for unique $u\in U$ and $b\in B$, and it follows that $u\in U(k)$ and $b\in B(k)$.  Likewise, multiplication gives $k$-isomorphisms of varieties
 $$ U^-\times T\times U\to U^-\times B\to U^-B= n_0(Bn_0B), $$
 so $U^-B$ is open and $(U^-B)(k)= U^-(k)B(k)= U^-(k)T(k)U(k)$.

\section{The algebraically closed case}
\label{sec:algclosed}

 Throughout this section $k$ is algebraically closed.  We need to recall some results from geometric invariant theory \cite[Ch.\ 3]{MR546290}.  Let $H$ be a reductive group acting on an affine variety $X$ over $\overline{k}$.  We denote the orbit of $x\in X$ by $H\cdot x$ and the stabiliser of $x$ by $H_x$.  One may form the affine quotient variety $X/H$.  The points of $X/H$ correspond to the closed $H$-orbits. We have a canonical projection $\pi_X\colon X\to X/H$.  The closure $\overline{H\cdot x}$ of any orbit $H\cdot x$ contains a unique closed orbit $H\cdot y$, and we have $\pi_X(x)= \pi_X(y)$.  If $C\subseteq X$ is closed and $H$-stable then $\pi_X(C)$ is closed.
 
 In particular, $H$ acts on itself by inner automorphisms---that is, by conjugation---and the orbit $H\cdot h$ is the conjugacy class of $h$.  We denote the quotient variety by $H/{\rm Inn}(H)$ and the canonical projection by $\pi_H\colon H\to H/{\rm Inn}(H)$.  If $h= h_sh_u$ is the Jordan decomposition of $h$ then $H\cdot h_s$ is the unique closed orbit contained in $\overline{H\cdot h}$; so $H\cdot h$ is closed if and only if $h$ is semisimple, and $\pi_H(h)= \pi_H(1)$ if and only if $h$ is unipotent.  Fix a maximal torus $T$ of $H$.  The Weyl group $W$ acts on $T$ by conjugation.  The inclusion of $T$ in $G$ gives rise to a map $\psi_T\colon T/W\to H/{\rm Inn}(H)$; it is well known that $\psi_T$ is an isomorphism of varieties.
 
Now assume $G$ is simply connected.  We can write $G\cong G_1\times\cdots \times G_m$, where the $G_i$ are simple.  Let $\nu_i\colon G\to G_i$ be the canonical projection.  Set $r_i= \rank(G_i)$ for $1\leq i\leq m$.

\begin{lemma}
\label{lem:nonunipt}
 Let $C$ be a closed $G$-stable subset of $G$ such that $C\not\subseteq Z(G)$.  Then there exist $g\in C$ and $x\in G$ such that $[g,x]$ is not unipotent. 
\end{lemma}

\begin{proof}
 Let $g\in C$ such that $g\not\in Z(G)$.  Note that $g_s\in C$ as $C$ is closed and conjugation-invariant.  If $g_s$ is not central in $G$ then we can choose a maximal torus $T'$ of $G$ such that $g_s\in T'$; then $[g_s,x]$ is a nontrivial element of $T$ for some $x\in N_G(T)$, and we are done.  So we can assume $g_s$ is central in $G$.  Then $g_u$ is a nontrivial unipotent element of $G$.
By \cite[Lem.\ 3.2]{MR2125071},
 $\overline{G\cdot g}$ contains an element of the form $g_su$, where $1\neq u$ belongs to some root group $U_\alpha$.  Let $n\in N_{G_\alpha}(T_\alpha)$ represent the nontrivial element of the Weyl group $N_{G_\alpha}(T_\alpha)/T_\alpha$.  Recall that $G_\alpha$ is isomorphic to $\SL_2$ or $\PGL_2$.  Explicit calculations with $2\times 2$ matrices (cf.\ the proof of Lemma~\ref{lem:max_tor} below) show that $[u,n]= [g_su,n]$ is not unipotent.  This completes the proof.
\end{proof}

Suppose we are given $G$-conjugacy classes $C_1,\ldots, C_m$ of $G$ such that for each $i$, $\nu_i(C_i)$ is noncentral in $G_i$ (we do not insist that the $C_i$ are all distinct).  Set $D_i= [C_i,G_i]$ and $E_i= \overline{D_i}= \overline{[\overline{C_i}, G_i]}$.  Note that for each $i$, $D_i$ is conjugation-invariant and constructible, and $D_i^{-1}= D_i$; likewise, $E_i$ is conjugation-invariant and irreducible, and $E_i^{-1}= E_i$.
  
\begin{proposition}
\label{prop:dense}
 Let $G$, etc., be as above, and set
 $X= D_1\cup\cdots \cup D_m$.  Then $B_X(r)$ contains a constructible dense subset of $G$.
\end{proposition}

\begin{proof}
 It suffices to prove that the constructible set
 $D_{i_1}\cdots D_{i_r}$ is dense in $G$ for some $i_1,\ldots, i_r$.  It is enough to show that
 the constructible set $E_{i_1}\cdots E_{i_r}$ is a dense subset of $G$ for some $i_1,\ldots, i_r$.
 
 Fix a maximal torus $T$ of $G$ and set $T_i= T\cap G_i$ for each $i$.  Clearly it is enough to prove that $(E_i)^{r_i}$ is a dense subset of $G_i$ for each $i$.  For notational convenience, we assume therefore that $m= 1$ and $G= G_1$ is simple; then $T= T_1$.  Set $C= C_1= \nu_1(C_1)$ and $E= E_1$; we prove that $E^r$ is a dense subset of $G$.  By hypothesis, $E= \overline{[\overline{C},G]}$ is an irreducible positive-dimensional subvariety of $G$.  Set $A= E\cap T$.  We claim that $A$ has an irreducible component $A'$ such that $\dim(A')> 0$.
 
 Set $F= \pi_G(E)$; note that $F$ is closed and irreducible because $E$ is closed, conjugation-invariant and irreducible.  Suppose $\dim(F)= 0$.  Since $1\in E$, we have $F= \{\pi_G(1)\}$, which forces $E$ to consist of unipotent elements.  But this is impossible by Lemma~\ref{lem:nonunipt}.  We deduce that $\dim(F)>0$.  Clearly $\pi_G(A)\subseteq F$.  Conversely, given $g\in E$, write $g= g_sg_u$ (Jordan decomposition).  Since $E$ is conjugation-invariant, we can, by conjugating $g$, assume without loss that $g_s\in T$.  We have $g_s\in \overline{G\cdot g}\cap T\subseteq A$ and $\pi_G(g_s)= \pi_G(g)$.  This shows that $F\subseteq \pi_G(A)$.  Hence $F= \pi_G(A)$.  Let $\pi_W\colon T\to T/W$ be the canonical projection.  Now $F':= \psi_T^{-1}(F)$ is an irreducible closed positive-dimensional subset of $T/W$, with $A= \pi_W^{-1}(F')$.  Since $W$ is finite, $\pi_W$ is a finite map and the fibres of $\pi_W$ are precisely the $W$-orbits.
 Hence the irreducible components of $A$ are permuted transitively by $W$, and each surjects onto $F'$.  Thus any irreducible component $A'$ of $A$ has the desired properties.
 
 Let $A_1,\ldots, A_t$ be the $W$-conjugates of $A'$.  The $A_i$ generate a nontrivial $W$-stable subtorus $S$ of $T$.  Hence the subset $V$ of $X(T)\otimes_{\mathbb Z} {\mathbb R}$ spanned by $\{\chi\in X(T)\,|\,\chi(S)= 1\}$ is proper and $W$-stable.  But $W$ acts absolutely irreducibly on $X(T)\otimes_{\mathbb Z} {\mathbb R}$,
 so $V= 0$.  This forces $S$ to be the whole of $T$.  So the $A_i$ generate $T$.  By the argument of \cite[Sec.\ 5]{KLM1} or \cite[7.5\ Prop.]{MR0396773}, there exist $i_1,\ldots, i_r\in \{1,\ldots, t\}$ and $\epsilon_1,\ldots, \epsilon_r\in \{\pm 1\}$ such that $A_{i_1}^{\epsilon_1}\cdots A_{i_r}^{\epsilon_r}$ is a constructible dense subset of $T$.  Hence $E^r$ contains a constructible dense subset of $T$, and we deduce that $E^r$ is a constructible dense subset of $G$.  This completes the proof.
\end{proof}

\begin{proof}[Proof of Theorem~\ref{thm:algclosed}]
 We have $\Delta(\widetilde{G})\leq \Delta(G)$ by \cite[Lem.\ 2.16]{KLM1}, where $\widetilde{G}$ is the simply connected cover of $G$.  Hence there is no harm in assuming $G$ is simply connected.  Let $X$ be a finite normal generating set for $G$.  We can choose $x_1,\ldots, x_m\in X$ such that $\nu_i(x_i)$ is noncentral in $G_i$ for $1\leq i\leq m$.  Let $C_i= G\cdot x_i\subseteq X$, let $D_i= [C_i,G]$ and let $X'= D_1\cup\cdots \cup D_m$.  By Proposition~\ref{prop:dense}, $B_{X'}(r)$ contains a dense constructible subset of $G$.  Since $D_i\subseteq B_{C_i}(2)$ for each $i$, $B_X(2r)$ contains a nonempty open subset $U$ of $G$.  Now $U^2= G$ by \cite[I.1.3\ Prop.]{MR1102012}, so $B_X(4r)\supseteq B_X(2r)B_X(2r)\supseteq U^2= G$.  It follows that $\Delta(G)\leq 4r$, as required. 
\end{proof}

\section{The isotropic case}
\label{sec:isotropic}

Now we consider the case of arbitrary semisimple $G$.  There is no harm in replacing $G$ with the Zariski closure of $G^+(k)$, which is the product of the isotropic $k$-simple factors of $G$.  Hence we assume in this section that $G^+(k)$ is dense in $G$.

We start by noting a corollary of Proposition~\ref{prop:dense}.
Let $X\subseteq G^+(k)$ such that $X$ is a finite normal generating set for $G$.  By Proposition~\ref{prop:dense}, there exist $i_1,\ldots, i_r\in \{1,\ldots, m\}$ such that the image of the map $f\colon G^{2r}\to G$ defined by
$$ f(h_1,\ldots, h_r,g_1,\ldots, g_r)= (h_1x_1h_1^{-1}g_1x_1^{-1}g_1^{-1})\cdots (h_rx_
rh_r^{-1}g_rx_r^{-1}g_r^{-1}) $$
contains a nonempty open subset $G'$ of $G$.  Now let $O$ be a nonempty open subset of $G$.  Then $f^{-1}(G'\cap O)$ is a nonempty open subset of $G^{2r}$.  But $G^+(k)$ is dense in $G$, so $G^+(k)^{2r}$ is dense in $G^{2r}$.  It follows that $f(h_1,\ldots, h_r,g_1,\ldots, g_r)\in O$ for some $h_1,\ldots, h_r$, $g_1,\ldots, g_r\in G^+(k)^{2r}$.  We deduce that for any nonempty open subset $O$ of $G$,
\begin{equation}
\label{eqn:k_dense}
 B_X(2r)\cap O\neq \emptyset.
\end{equation}
 
\begin{remark}
 Let $C= {\rm im}(f)$, where $f$ is as above.  It follows from Eqn.~(\ref{eqn:k_dense}) and Lemma~\ref{lem:open_prod} that $C(k)^2= G(k)$.  We cannot, however, conclude directly from this that $B_X(2r)^2= G(k)$: the problem is that although the map $f\colon G^{2r}\to C$ is surjective on $\overline{k}$-points, it need not be surjective on $k$-points.
\end{remark}

\begin{lemma}
\label{lem:reg_ss_par}
 There exists $t\in P(k)$ such that $t$ is regular semisimple.
\end{lemma}

\begin{proof}
 Define $f\colon G\times P\to G$ by $f(g,h)= ghg^{-1}$.  Then $f$ is surjective since every element of $G$ belongs to a Borel subgroup of $G$.  Let $O$ be the set of regular semisimple elements of $G$, a nonempty open subset of $G$.  By \cite[Thm.\ 21.20(ii)]{MR1102012}, $P(k)$ is dense in $P$, and we know that $G(k)$ is dense in $G$, so $G(k)\times P(k)$ is dense in $G\times P$.  It follows that there is a point $(g,t)\in (G(k)\times P(k))\cap f^{-1}(O)$.  Then $gtg^{-1}$ is regular semisimple, so $t\in P(k)$ is regular semisimple also.
\end{proof}

\begin{lemma}
\label{lem:reg_ss}
 Let $t\in P(k)$ be regular semisimple.  Then $U(k)\subseteq B_t(2)$.
\end{lemma}

\begin{proof}
 Define $f\colon U\to U$ by $f(u)= utu^{-1}t^{-1}$.  The conjugacy class $U\cdot t$ is closed because orbits of unipotent groups are closed,
 so ${\rm im}(f)$ is a closed subvariety of $U$.  Since $t$ is regular, it is easily checked that $f$ is injective and the derivative $df_u$ is an isomorphism for each $u\in U$.  It follows from Zariski's Main Theorem that $f$ is an isomorphism of varieties.  As $f$ is defined over $k$, $f$ gives a bijection from $U(k)$ to $U(k)$, and the result follows.
\end{proof}

\begin{lemma}
\label{lem:gengen}
 Let $X$ be a finite normal generating subset for $G^+(k)$.  Then $X$ normally generates $G$.
\end{lemma}

\begin{proof}
 There exists $d\in {\mathbb N}$ such that $(G(k)\cdot X)^d= G^+(k)$.  So the constructible set $(G\cdot X)^d$ contains $G^+(k)$ and is therefore dense in $G$.  This implies that $(G\cdot X)^d$ contains a nonempty open subset of $G$, so $(G\cdot X)^d(G\cdot X)^d= G$.  Hence $X$ is a finite normal generating set for $G$.
\end{proof}

\begin{proposition}
\label{prop:gettingU}
 Let $X$ be a finite subset of $G^+(k)$ such that $X$ normally generates $G$.  Then $U(k)\subseteq B_X(8r)$.
\end{proposition}

\begin{proof}
 The big cell $Pn_0P$ is open, so by Eqn.~(\ref{eqn:k_dense}), we can choose $g\in B_X(2r)\cap Pn_0P$.  We can write $g= xn_0x'$ for some $x,x'\in P(k)$.  Since $B_X(2r)$ is conjugation-invariant, there is no harm in replacing $g$ with $(x')^{-1}gx'$, so we can assume that $x'= 1$ and $g= xn_0$.  Let $C_1= \{n_0x_1\,|\,x_1\in P, xn_0^2x_1\ \mbox{is regular semisimple}\}$.  Let $O_1= P\cdot C_1= U\cdot C_1$; then $O_1$ is a constructible dense subset of $G$.
   By Eqn.~(\ref{eqn:k_dense}), there exists $g\in B_X(2r)\cap O_1$.  We can write $g= un_0x_1u^{-1}$ where $xn_0^2x_1$ is regular semisimple and $u\in U$.  Since $g\in G(k)$, both $u$ and $x_1u^{-1}$ belong to $G(k)$.  Hence $n_0x_1\in B_X(2r)\cap C_1$.  It follows that $t:= xn_0^2x_1$ is regular semisimple and belongs to $B_X(4r)$.  We have $t\in B_X(4r)\cap P(k)$, so $U(k)\subseteq B_t(2)\subseteq B_X(8r)$ by Lemma~\ref{lem:reg_ss}.  This completes the proof.
\end{proof}

\begin{proof}[Proof of Theorem~\ref{thm:bdd_unifbdd}]
 Suppose $G(k)= G^+(k)$.   By Lemma~\ref{lem:reg_ss_par}, there exists $t\in P(k)$ such that $t$ is regular semisimple.  By Lemma~\ref{lem:reg_ss}, $B_t(2)$ contains $U(k)$.  Since $G(k)$ is generated by the $G(k)$-conjugates of $U(k)$, we deduce that $\{t\}$ normally generates $G(k)$.  Hence $G(k)$ is finitely normally generated.
 
 Now suppose further that $G(k)$ is bounded.  Fix a finite normal generating set $Y$ for $G(k)$.  Then $G(k)= B_Y(s)$ for some $s\in {\mathbb N}$ and $Y\subseteq B_{U(k)}(d)$ for some $d\in {\mathbb N}$.  Let $X$ be any finite normal generating set for $G(k)$.  Then $X$ is normally generates $G$ by Lemma~\ref{lem:gengen}.  By Proposition~\ref{prop:gettingU}, $U(k)\subseteq B_X(8r)$.  So
 $$ G(k)= B_Y(s)\subseteq B_{U(k)}(sd)\subseteq B_X(8rsd). $$
 This shows that $G(k)$ is uniformly bounded, as required.
\end{proof}

\begin{remark}
 The hypothesis that $G^+(k)= G(k)$ holds in many cases if $G$ is $k$-simple and simply connected---this is the content of the Kneser-Tits conjecture, which holds, for example, when $k$ is a local field.
\end{remark}

\begin{example}
 It is well known that the abelianisation of $\SO_3({\mathbb Q})$ is ${\mathbb Q}^*/({\mathbb Q}^*)^2$, which is an infinitely generated abelian group.  It follows that $\SO_3({\mathbb Q})$ is not finitely normally generated.  Note that $\SO_3^+({\mathbb Q})= 1$ since $\SO_3$ is anisotropic over ${\mathbb Q}$.
\end{example}

\section{The split case}
\label{sec:split}

In this section we assume $G$ is split over $k$.
If $G$ is simply connected then the Kneser-Tits Conjecture holds for $G$, so $G^+(k)= G(k)$ in this case.

\begin{lemma}
\label{lem:max_tor}
 Suppose $(*)$ each $G_\alpha$ is isomorphic to $\SL_2$.  Let $t_i\in T_{\alpha_i}(k)$ for $1\leq i\leq r$ and set $t= t_1\cdots t_r$.  There exist $u_i, w_i\in U_{\alpha_i}(k)$ and $v_i, x_i\in U_{-\alpha_i}(k)$ for $1\leq i\leq r$ such that $t= x_r\cdots x_1u_r\cdots u_1v_1\cdots v_rw_1\cdots w_r$.
\end{lemma}

\begin{proof}
 We use induction on $r$.  The case $r=0$ is vacuous.  Now consider the case $r=1$.  Then $G\cong \SL_2$.  For any $a,b,c,d\in k$ we have
 $$ \left(\begin{smallmatrix} 1 & a \\ 0 & 1 \end{smallmatrix} \right)\left(\begin{smallmatrix} 1 & 0 \\ b & 1 \end{smallmatrix} \right) \left(\begin{smallmatrix} 1 & c \\ 0 & 1 \end{smallmatrix} \right)\left(\begin{smallmatrix} 1 & 0 \\ d & 1 \end{smallmatrix} \right)= \left(\begin{smallmatrix} 1+ab & a \\ b & 1 \end{smallmatrix} \right)\left(\begin{smallmatrix} 1+cd & c \\ d & 1 \end{smallmatrix} \right)= \left(\begin{smallmatrix} 1+ab+cd+abcd+ad & c+abc+a \\ b+bcd+d & bc+1 \end{smallmatrix} \right). $$
 Let $x\in k^*$.  Set $a= -x$, $b= x^{-1}-1$, $c= 1$ and $d= x-1$; then the matrix above becomes $\left(\begin{smallmatrix} x & 0 \\ 0 & x^{-1} \end{smallmatrix} \right)$.  Hence the result holds when $r=1$.
 
 Now suppose $r>1$.  Let $H$ be the semisimple group with root system spanned by $\pm \alpha_1,\ldots, \pm \alpha_{r-1}$.  Clearly condition $(*)$ holds for $H$.  Let $s= t_1\cdots t_{r-1}$.  By our induction hypothesis, there exist $u_i, w_i\in U_{\alpha_i}(k)$ and $v_i, x_i\in U_{-\alpha_i}(k)$ for $1\leq i\leq r-1$ such that
 $$ s= x_{r-1}\cdots x_1u_{r-1}\cdots u_1v_1\cdots v_{r-1}w_1\cdots w_{r-1}. $$
 By the $\SL_2$ case considered above, $t_r= x_r'u_r'v_r'w_r'$ for some $u_r, w_r\in U_{\alpha_r}$ and some $v_r,x_r\in U_{-\alpha_r}$.  Set $x_r= sx_r's^{-1}$, $u_r= su_r's^{-1}$, $v_r= v_r'$ and $w_r= w_r'$.  We have
 \begin{eqnarray*}
  & & x_rx_{r-1}\cdots x_1u_ru_{r-1}\cdots u_1v_1\cdots v_{r-1}v_rw_1\cdots w_{r-1}w_r \\
   & = & x_ru_rx_{r-1}\cdots x_1u_{r-1}\cdots u_1v_1\cdots v_{r-1}w_1\cdots w_{r-1}v_rw_r \\
   & = & x_ru_rsv_rw_r \\
   & = & sx_r'u_r'v_r'w_r' \\
   & = & st_r= t.
 \end{eqnarray*}
 The result follows by induction.
\end{proof}

\begin{proposition}
\label{prop:RuB}
 Suppose $G$ is simply connected.  Let $X\subseteq G^+(k)$ such that $U(k)\subseteq X$.  Then $B_X(7)= G^+(k)$. 
\end{proposition}

\begin{proof}
 Since $G$ is simply connected, $(*)$ holds for $G$ and the map $\psi\colon {\mathbb G}_m^r\to T$ given by $\psi(a_1,\ldots, a_r)= \alpha_1^\vee(a_1)\cdots \alpha_r^\vee(a_r)$ is a $k$-isomorphism.
 It follows that $T(k)= T_{\alpha_1}(k)\cdots T_{\alpha_r}(k)$, so $T(k)\subseteq B_X(4)$ by Lemma~\ref{lem:max_tor}.  Hence $U^-(k)B(k)= U^-(k)T(k)U(k)\subseteq B_X(1)B_X(4)B_X(1)\subseteq B_X(6)$.  Now $G(k)= (U^-B)^{-1}(k)(U^-B)(k)$ by Lemma~\ref{lem:open_prod}.  But
 $$ (U^-B)^{-1}(k)(U^-B)(k)= B(k)U^-(k)U^-(k)B(k)= U(k)T(k)U^-(k)T(k)U(k) $$
 $$ = U(k)U^-(k)T(k)U(k)= U(k)U^-(k)B(k)\subseteq B_X(1)B_X(6)\subseteq B_X(7), $$
 so we are done.
\end{proof}

\begin{proof}[Proof of Theorem~\ref{thm:split}]
 Let $\widetilde{G}$ be the split form of the simply connected cover of $G$ and let $\psi\colon \widetilde{G}\to G$ be the canonical projection.  Then $\psi$ is a $k$-defined central isogeny, so by \cite[V.22.6\ Thm.]{MR1102012}, the map $\widetilde{B}\mapsto \psi(\widetilde{B})$ gives a bijection between the set of $k$-Borel subgroups of $\widetilde{G}$ and the set of $k$-Borel subgroups of $G$; moreover, for each $\widetilde{B}$, $\psi$ gives rise to a $k$-isomorphism from $R_u(\widetilde{B})$ to $R_u(B)$ \cite[Prop.\ V.22.4]{MR1102012}.
 It follows that $\psi(\widetilde{G}^+(k))= G^+(k)$.  By \cite[Lem.\ 2.16]{KLM1} we have $\Delta(G^+(k))\leq \Delta(\widetilde{G}^+(k))$, so we can assume without loss that $G$ is simply connected.  In particular, $G^+(k)= G(k)$.
 
 Let $X$ be a finite normal generating set for $G(k)$.  Then $X$ is a finite normal generating set for $G$ (Lemma~\ref{lem:gengen}), so by Eqn.\ (\ref{eqn:k_dense}) there exists $t\in B_X(2r)$ such that $t$ is regular semisimple.  We have $U(k)\subseteq B_t(2)$ by Lemma~\ref{lem:reg_ss} and $G(k)\subseteq B_{U(k)}(7)$ by Proposition~\ref{prop:RuB}.  So
 $$ G(k)\subseteq B_{U(k)}(7)\subseteq B_t(14)\subseteq B_X(28r). $$
This shows that $\Delta(G(k))\leq 28r$, as required.
\end{proof}

\begin{example}
\label{ex:lowerbd}
 (a) Let $G= \SL_n(k)$ where $n\geq 3$, let $g$ be the elementary matrix $E_{1n}(1)$ and let $X= G(k)\cdot g$.
By \cite[Prop.\ 6.23]{KLM1}, $X$ generates $G(k)$.  One sees easily by direct computation that the centraliser $C_G(g)$ has dimension $n^2- 2n+ 1$, so $\dim(G\cdot g)= 2n-2$.  A simple dimension-counting argument shows that if $t< \frac{1}{2}\rank(G)$ then $\overline{X^t}$ is a proper closed subvariety of $G$.  Since $G(k)$ is dense in $G$, it follows that $\overline{X^t}$ does not contain $G(k)$, so $X(k)^t$ does not contain $G(k)$.  We deduce that $\Delta(G(k))\geq \frac{1}{2}\rank(G)$.

 (b) The bounds in Theorems~\ref{thm:algclosed} and \ref{thm:split} are far from sharp.  Aseeri has shown by direct calculation that $3\leq \Delta(\SL_3({\mathbb C}))\leq 6$ and that $\Delta(\SL_2({\mathbb C})^m)= 3m$ and $\Delta(\PGL_2({\mathbb C})^m)= 2m$ for every $k\in {\mathbb N}$ \cite[Thm.\ 8.0.2, Thm.\ 7.2.10, Thm.\ 7.2.6]{aseeri}, whereas Theorem~\ref{thm:algclosed} yields the bounds $\Delta(\SL_3({\mathbb C}))\leq 8$ and $\Delta(\SL_2({\mathbb C})^m), \Delta(\PGL_2({\mathbb C})^m)\leq 4m$.  Aseeri also showed that $3\leq \Delta(\SL_3({\mathbb R}))\leq 4$, whereas Theorem~\ref{thm:split} gives $\Delta(\SL_3({\mathbb R}))\leq 56$.
\end{example}

\section{Semisimple Lie groups}
\label{sec:Lie}

\begin{proof}[Proof of Theorems~\ref{thm:real_Lie} and \ref{thm:complex_Lie}]
 Let $H$ be a linear semisimple Lie group such that $H$ has no compact simple factors.  By \cite[Thm.\ III.2.13]{milne}, there is a complex semisimple algebraic group $G$ defined over ${\mathbb R}$ such that $G^+({\mathbb R})= H$.  Now $Z(H)$ is finite, so $H$ is finitely normally generated and bounded by \cite[Thm.~1.2]{KLM1}.  It follows from Theorem~\ref{thm:bdd_unifbdd} that $H$ is uniformly bounded.  If $H$ is split then $G$ is split over ${\mathbb R}$, so $\Delta(H)\leq 28\rank(H)$ by Theorem~\ref{thm:split}.
 
 The argument for the complex case is similar: if $H$ is a semisimple linear complex Lie group then there is a semisimple complex algebraic group $G$ such that the complex Lie group associated to $G$ is $H$ (cf.\ \cite[Ch.\ 4, Sec.\ 2, Problem 12]{MR1064110}, and $G$ is isomorphic to $H$.  The result now follows from Theorem~\ref{thm:algclosed}.
\end{proof}

\bibliography{bibliography}
\bibliographystyle{plain}

\end{document}